\documentclass[12pt,reqno]{amsart}

\usepackage{times}
\usepackage{graphicx}
\usepackage{xcolor}
\graphicspath{ {./images/} }

\usepackage{amsmath,amsthm,amsfonts,amscd,amssymb}
\usepackage{pst-node}
\usepackage{pst-plot}
\usepackage[all]{xy}
\usepackage{stackengine}

\usepackage{tikz, braids}
\usetikzlibrary{decorations.pathreplacing, arrows}

\usepackage[
    a4paper,
    left=22mm,
    right=22mm,
    top=25mm,
    bottom=25mm
]{geometry}

\usepackage{setspace}

\newtheorem{theorem}{Theorem}[section]
\newtheorem{corollary}[theorem]{Corollary}

\newtheorem{lemma}[theorem]{Lemma}

\newtheorem{proposition}[theorem]{Proposition}
\newtheorem{example}[theorem]{Example}

\newtheorem*{maintheorem1}{Main Result 1}
\newtheorem*{maintheorem2}{Main Result 2}
\newtheorem*{maintheorem3}{Main Result 3}

\numberwithin{theorem}{section}
\numberwithin{equation}{section}

\newcommand{\myop}[1]{\operatorname{#1}}

\newcommand{\id}{\myop{id}}

\newcommand{\norm}[1]{\left\Vert#1\right\Vert}

\newcommand{\la}{\langle}
\newcommand{\ra}{\rangle}

\newcommand{\Comp}{\mathbb{C}}

\newcommand{\g}{\mathbb{G}}

\newcommand{\n}{\mathbb{N}}

\newcommand{\z}{\mathbb{Z}}

\global\long\def\tp{\mathop{\xymatrix{*+<.7ex>[o][F-]{\scriptstyle \top}}
 } }

\begin{document}

\title[Non-unitarizable representations of compact and discrete quantum groups]{Non-unitarizable representations of compact and discrete quantum groups}

\author[Fatemeh Khosravi]{Fatemeh Khosravi} 
\address{Fatemeh Khosravi, Department of Pure Mathematics, Faculty of Mathematics and Statistics, University of Isfahan, Isfahan 81746-73441, Iran} \email{f.khosravi@mcs.ui.ac.ir}

\author[Sang-Gyun Youn]{Sang-Gyun Youn}
\address{Sang-Gyun Youn, 
Department of Mathematics Education and NextQuantum Innovation Research Center, Seoul National University, 
Gwanak-Ro 1, Gwanak-Gu, Seoul 08826, Republic of Korea}
\email{s.youn@snu.ac.kr }

\begin{abstract}

We study the similarity problem for representations of compact and discrete quantum groups. We first prove that every non-degenerate contractive representation of a compact or discrete quantum group is automatically unitary. Then, for compact quantum groups, we develop an interpolation method for constructing explicit non-unitarizable non-degenerate representations with norms arbitrarily close to $1$. In particular, this applies to all non-Kac compact quantum groups whose dual has subexponential growth, including the Drinfeld--Jimbo quantum groups $G_q$, as well as to the non-Kac free unitary quantum groups $U_F^+$ with $F\in\operatorname{GL}_2(\Comp)$. On the discrete quantum group side, we establish a quantum analogue of the classical lifting principle for non-unitarizable uniformly bounded representations. Combining this lifting principle with recent results on maximal Kac quantum subgroups, we construct explicit non-unitarizable non-degenerate representations with norms arbitrarily close to $1$ for all unitary free quantum groups $\mathbb FU_F$ and all non-amenable orthogonal free quantum groups $\mathbb FO_F$.

\end{abstract}

\maketitle


\section{Introduction}\label{intro}

The similarity problem for representations of locally compact groups has a long history. Let $G$ be a locally compact group. A strongly continuous homomorphism $\pi:G\longrightarrow \operatorname{GL}(H_{\pi})$ is called a \textit{representation}. It is said to be \textit{uniformly bounded} if $\displaystyle \norm{\pi}= \sup_{x\in G}\norm{\pi(x)}_{B(H_{\pi})}<\infty$, and is called a \textit{unitary representation} if $\pi(x)$ is unitary for all $x\in G$. A representation $\pi$ is called
\textit{unitarizable} if there exists an invertible operator $T\in B(H_{\pi})$ such that
\begin{equation}
\pi_T(x)=T\pi(x)T^{-1}, \qquad x\in G,
\end{equation}
defines a unitary representation.

\setstretch{1.1}

A classical result of Day \cite{Da50} and Dixmier \cite{Di50} states that every uniformly bounded representation of an amenable locally compact group is unitarizable. The famous {\it Dixmier similarity problem} asks whether the converse holds: does every non-amenable locally compact group admit a non-unitarizable uniformly bounded representation?

\setstretch{1.1}

Although this problem remains open in general, considerable effort has been devoted to constructing explicit non-unitarizable uniformly bounded representations. Important contributions include works on Lie groups
\cite{EhMa55,KuSt60,KuSt61,KuSt67,KuSt73,Lip67,Lip69,Lip69b}, free groups and free products \cite{MaZa83,MaZa83b,PySz86,BoFe84,Boz87,Szw89,Szw90,Wys93}, and several other classes of groups \cite{Sal68,Sal67,EpMo09,Osi09,MoOz10}. We also refer to \cite{Cow77,Cow82,Pis01,Pis05} for further background on uniformly bounded representations and similarity problems.

\setstretch{1.1}

The similarity problem has also been studied in the framework of {\it locally compact quantum groups} \cite{BrSa10,ChSa13,BrDaSa13,LeSaSp16,LSS19}. A representation of a locally compact quantum group $\g=(L^{\infty}(\g),\Delta,\varphi,\psi)$ on a Hilbert space $H_V$ is an element $V\in L^{\infty}(\g)\overline{\otimes}B(H_V)$ satisfying $(\Delta\otimes\mathrm{id})(V)=V_{13}V_{23}$, where we use the leg notation. A representation $V$ is called \textit{contractive} if
\begin{equation}
    \norm{V}_{L^{\infty}(\g)\overline{\otimes}B(H_V)}\leq 1,
\end{equation}
and is called \textit{unitarizable} if there exists an invertible operator $T\in B(H_V)$ such that
\begin{equation}
    V_T=(1\otimes T)V(1\otimes T^{-1})
\end{equation}
is unitary in $L^{\infty}(\g)\overline{\otimes}B(H_V)$.

\setstretch{1.1}

Following \cite{BrYo19}, the similarity problem for locally compact quantum groups naturally splits into the following two questions:
\begin{enumerate}
    \item Is every non-degenerate contractive representation of $\g$
    automatically unitary?
    \item Is every non-degenerate representation of $\g$ unitarizable?
\end{enumerate}

Question (1) is already nontrivial in the cocommutative case $\g=\widehat{G}$. In this setting, the question was raised in
\cite{EfRu03,BrSa10}, and it is known that every non-degenerate contractive representation of $\widehat{G}$ is unitary when the connected component of the identity $G_e$ is amenable; see \cite{LeSaSp16}. The same conclusion follows for compact quantum groups of Kac type from the proof of \cite[Theorem 6.2]{BrDaSa13}, and for amenable discrete quantum groups of Kac type from \cite[Section 4.1]{BrYo19}. Beyond these cases, Question (1) had remained open for non-Kac compact quantum groups and for non-amenable discrete quantum groups.

\setstretch{1.1}

Our first main result answers Question (1) affirmatively for all compact and discrete quantum groups. The proofs are direct and use standard tools from locally compact quantum group theory.

\begin{maintheorem1}
Let $\g$ be an arbitrary compact quantum group or discrete quantum group. If a non-degenerate representation $V\in L^{\infty}(\g)\overline{\otimes}B(H_V)$ satisfies $\norm{V}_{L^{\infty}(\g)\overline{\otimes}B(H_V)} \leq 1$, then $V$ is unitary.
\end{maintheorem1}

This result shows that Question (1) always has an affirmative answer for compact and discrete quantum groups. Thus, within these classes, Question (1) no longer imposes any restriction, and the similarity problem reduces to Question (2). It was already known that Question (2) has an affirmative answer for compact quantum groups of Kac type and for amenable discrete quantum groups of Kac type; see \cite[Theorem 6.2]{BrDaSa13} and \cite[Section 4.1]{BrYo19}:
\begin{gather} 
\text{\normalfont\normalsize $\g$ is a compact Kac quantum group or an amenable discrete Kac quantum group} \notag\\[3pt] \Longrightarrow ~\text{\normalfont\normalsize Question (2) has an affirmative answer.} \label{eq-mainquestion} 
\end{gather}

In the spirit of the Dixmier similarity problem, we study \eqref{eq-mainquestion} in the converse direction from a constructive point of view. More precisely, we consider the following question:
\begin{enumerate}
    \item[(2)'] If $\g$ is a non-Kac compact quantum group or a
    non-amenable discrete quantum group, can one explicitly construct a non-unitarizable non-degenerate representation of $\g$?
\end{enumerate}

Since this question subsumes the Dixmier similarity problem for non-amenable discrete groups, we do not aim for a complete solution. For non-Kac compact quantum groups, the principal general result in this direction is \cite[Theorem 3.1]{BrYo19}. It shows that if $\g$ is non-Kac and its dual has subexponential growth, as happens for the Drinfeld--Jimbo quantum groups $G_q$, then at least one of Questions (1) and (2) must have a negative answer. In view of {\bf Main Result 1}, it follows that Question (2) necessarily fails for every compact quantum group in this class.

\setstretch{1.1}

The argument of \cite[Theorem 3.1]{BrYo19}, however, establishes only the existence of a counterexample. It relies on the general methods developed in \cite{Pi98,Sp02} and does not specify how such a representation may be chosen. An explicit non-unitarizable non-degenerate representation is constructed there for a specially designed compact quantum group $\g=\prod_{n=1}^{\infty}SU_{q_n}(2)$ with $q_n\in(0,1)$ such that $\displaystyle \lim_{n\rightarrow\infty}q_n=0$. For individual examples such as $\g=SU_q(2)$, their arguments show the existence but do not provide an explicit non-unitarizable non-degenerate representation.

\setstretch{1.1}
Our second main result establishes a general criterion for constructing explicit non-unitarizable non-degenerate representations of compact quantum groups with norms arbitrarily close to $1$. The construction is based on interpolating between the contragredient representation $(u^{\alpha})^c$ and its unitarized form $(1\otimes Q_{\alpha}^{\frac{1}{2}})(u^{\alpha})^c(1\otimes Q_{\alpha}^{-\frac{1}{2}})$. More precisely, for a subset $E\subseteq\operatorname{Irr}(\g)$ and a function $\theta:E\rightarrow[0,1]$, we consider the direct sum of the intermediate representations
\begin{equation}
v_{E,\theta}=\bigoplus_{\alpha\in E}(1\otimes Q_{\alpha}^{\frac{\theta(\alpha)}{2}})(u^{\alpha})^c(1\otimes Q_{\alpha}^{-\frac{\theta(\alpha)}{2}}).
\end{equation}
We show that $\displaystyle \norm{v_{E,\theta}} \leq \sup_{\alpha\in E}n_{\alpha}^{1-\theta(\alpha)}$ and, moreover, that $v_{E,\theta}$ is unitarizable if and only if
\begin{equation}
\sup_{\alpha\in E}\left(\norm{Q_{\alpha}}_{op}\norm{Q_{\alpha}^{-1}}_{op}\right)^{\frac{1-\theta(\alpha)}{2}}<\infty.
\end{equation}
By choosing suitable parameters $\theta(\alpha)$, we establish the following criterion.

\begin{maintheorem2}
Let $\g$ be a compact quantum group. Suppose that there exists a subset $E\subseteq\operatorname{Irr}(\g)$ such that
\begin{equation}
\sup_{\alpha\in E}\frac{\log\left(\norm{Q_{\alpha}}_{op}\norm{Q_{\alpha}^{-1}}_{op}\right)}{\log(1+n_{\alpha})}=\infty.
\end{equation}
Then, for any $\epsilon>0$, one can construct a non-unitarizable non-degenerate representation $V\in L^{\infty}(\g)\overline{\otimes}B(H_V)$ satisfying $\norm{V}_{L^{\infty}(\g)\overline{\otimes}B(H_V)}<1+\epsilon$. In particular, this conclusion holds if either
\begin{enumerate}
\item[(A)] $\g$ is non-Kac and its dual $\widehat{\g}$ has subexponential growth; or
\item[(B)] $\g$ is a non-Kac free unitary quantum group $U_F^+$ for some $F\in\operatorname{GL}_2(\Comp)$.
\end{enumerate}
\end{maintheorem2}

\setstretch{1.1}

We next turn our attention to discrete quantum groups. We first establish a quantum analogue of the classical lifting principle for uniformly bounded representations along quotient homomorphisms. More precisely, if $\widehat{\mathbb H}$ is a closed quantum subgroup of $\widehat{\mathbb G}$ in the sense of Vaes, then every representation of $\mathbb H$ can be lifted to a representation of $\mathbb G$ without changing either its norm or its unitarizability. This provides a general mechanism for transferring non-unitarizable non-degenerate representations to larger locally compact quantum groups.

We apply this lifting principle to two fundamental families of non-amenable discrete quantum groups, namely the unitary free quantum groups $\mathbb FU_F$, with $F\in\operatorname{GL}_N(\Comp)$ and $N\geq 2$, and the orthogonal free quantum groups $\mathbb FO_F$, with $F\in\operatorname{GL}_N(\Comp)$, $N\geq 3$. A key ingredient is the description of {\it maximal Kac quantum subgroups} for free quantum groups; see, for example, \cite{DaFrSk21}. Combining this description with uniformly bounded non-unitarizable representations arising from free groups, we establish the following theorem. This result is new even in the Kac case, namely for $\mathbb F U_N$ and $\mathbb F O_N$.

\begin{maintheorem3}
Let $\epsilon>0$. The following assertions hold.
\begin{enumerate}
\item Let $N\geq 2$ and $F\in\operatorname{GL}_N(\Comp)$. The unitary free quantum group $\mathbb FU_F$ admits a non-unitarizable non-degenerate representation of norm less than $1+\epsilon$.
\item Let $N\geq 3$ and $F\in\operatorname{GL}_N(\Comp)$ such that $F\overline{F}\in\mathbb{R}\cdot \operatorname{Id}_N$. The orthogonal free quantum group $\mathbb FO_F$ admits a non-unitarizable non-degenerate representation of norm less than $1+\epsilon$.
\end{enumerate}
\end{maintheorem3}

\section{Preliminaries}\label{pre}
\subsection{Locally compact quantum groups and their representations}

We begin by recalling the von Neumann algebraic formulation of locally compact quantum groups; see \cite{KuVa00,KuVa03}. A \textit{ locally compact quantum group} $\g= \left( L^{\infty}(\g),\Delta,\varphi,\psi \right)$ consists of a von Neumann algebra $L^{\infty}(\g)$, a normal unital $*$-homomorphism $\Delta:L^{\infty}(\g)\rightarrow L^{\infty}(\g) \overline{\otimes} L^{\infty}(\g)$ satisfying 
\begin{equation}
    (\Delta\otimes\operatorname{id})\Delta=(\operatorname{id}\otimes\Delta)\Delta,
\end{equation}
and normal semifinite faithful weights $\varphi,\psi$ satisfying
\begin{align}
    \varphi\left( (\omega\otimes\operatorname{id})\Delta(x)\right) &=\omega(1)\varphi(x),\\
    \psi\left((\operatorname{id}\otimes\omega)\Delta(y)\right)&=\omega(1)\psi(y),
\end{align}
for all $\omega\in L^{1}(\g)_{+}$, $x\in\mathfrak{M}_{\varphi}^{+}$, and $y\in\mathfrak{M}_{\psi}^{+}$. Here, $L^1(\g)$ is the predual of $L^{\infty}(\g)$, i.e. $L^1(\g)=L^{\infty}(\g)_*$ and $\mathfrak{M}_{\varphi}^{+}=\left\{x\in L^{\infty}(\g)_{+}:\varphi(x)<\infty \right\}$
and $\mathfrak{M}_{\psi}^{+}=\left\{y\in L^{\infty}(\g)_{+}:\psi(y)<\infty \right\}$.

The predual $L^{1}(\g)$ is a completely contractive Banach algebra with respect to the
convolution product 
\begin{equation}
    (\omega\star\eta)(x)=(\omega\otimes\eta)(\Delta(x)),\qquad \omega,\eta\in L^{1}(\g), x\in L^{\infty}(\g).
\end{equation}

A \textit{representation} of $\g$ on a Hilbert space $H_{V}$ is an element $V\in L^{\infty}(\g)\overline{\otimes}B(H_{V})$ satisfying
\begin{equation}
    (\Delta\otimes\operatorname{id})(V)=V_{13}V_{23}.
\end{equation}
Here and throughout the paper, we use the standard leg-numbering notation. 

A representation $V$ is called \textit{non-degenerate} if 
\begin{equation}
    \overline{\operatorname{span}\left\{(\omega\otimes \operatorname{id})(V)\xi:\omega\in L^{1}(\g), \xi\in H_{V}\right\}}=H_{V}.
\end{equation}
In particular, any uniformly bounded representation of a locally compact group is non-degenerate.

A representation $V$ is called \textit{ contractive} if $\norm{V}_{L^{\infty}(\g)\overline{\otimes}B(H_{V})}\leq 1$, and is called \textit{ isometric} if $V^{*}V=1_{L^{\infty}(\g)}\otimes I_{H_{V}}$. It is called \textit{ unitary} if
\begin{equation}
    V^{*}V=VV^{*}=1_{L^{\infty}(\g)}\otimes I_{H_{V}}.
\end{equation}
A representation $V$ is called \textit{unitarizable} if there exists an invertible operator $T\in B(H_{V})$ such that
\begin{equation}
    V_{T}=(1\otimes T)V(1\otimes T^{-1})
\end{equation}
is unitary.

For a family of representations $V_i\in L^{\infty}(\g)\overline{\otimes}B(H_i)$ such that $\displaystyle \sup_{i\in I}\norm{V_i}<\infty$, their
\textit{direct sum representation} is
\begin{equation}
    \bigoplus_{i\in I}V_i =\sum_{i\in I}(1\otimes t_i)V_i(1\otimes t_i^*)\in L^{\infty}(\g)\overline{\otimes}B\left(\bigoplus_{i\in I}H_i\right),
\end{equation}
where $t_i:H_i\hookrightarrow \bigoplus_{i\in I}H_i$ is the canonical isometry and the sum converges in the strong operator topology.

For representations $V\in L^{\infty}(\g)\overline{\otimes}B(H_{V})$ and $W\in L^{\infty}(\g)\overline{\otimes}B(H_{W})$, their \textit{tensor product representation} is defined by
\begin{equation}
    V\tp W=V_{12}W_{13} \in L^{\infty}(\g)\overline{\otimes}B(H_{V}\otimes H_{W}).
\end{equation}
The space of \textit{intertwiners} from $V$ to $W$ is
\begin{align}
\operatorname{Mor}(V,W)=\left\{A\in B(H_{V},H_{W}):(1\otimes A)V=W(1\otimes A)\right\}.
\end{align}
The representations $V$ and $W$ are called \textit{equivalent} (resp. unitarily equivalent) if $\operatorname{Mor}(V,W)$ contains an invertible (resp. unitary) operator. A representation $V$ is called \textit{irreducible} if $\operatorname{Mor}(V)=\operatorname{Mor}(V,V)=\Comp \cdot \operatorname{Id}_{H_V}$.

\subsection{Compact quantum groups}
A locally compact quantum group $\g=\left( L^{\infty}(\g),\Delta,\varphi,\psi \right)$ is called \textit{compact} if $\varphi(1_{L^{\infty}(\g)})<\infty$. In this case, $\varphi$ and $\psi$ are finite and therefore extend uniquely to normal faithful positive linear functionals on $L^{\infty}(\g)$. After normalization, these functionals coincide, and the resulting state $h$ is called the \textit{Haar state}. For any $a\in L^{\infty}(\g)$, we have
\begin{align}
    (\operatorname{id}\otimes h)(\Delta(a))=h(a)1_{L^{\infty}(\g)}=(h\otimes \operatorname{id})(\Delta(a)).
\end{align}
A compact quantum group $\g$ is called \textit{ Kac type} if $h(ab)=h(ba)$ for all $a,b\in L^{\infty}(\g)$.
Let $\operatorname{Irr}(\g)$ denote the set of equivalence classes of
irreducible unitary representations of $\g$, and for each
$\alpha\in\operatorname{Irr}(\g)$ we take a representative irreducible unitary
representation
\begin{align}
    u^{\alpha}=\sum_{i,j=1}^{n_{\alpha}}u^{\alpha}_{ij}\otimes e^{\alpha}_{ij} \in L^{\infty}(\g)\overline{\otimes}B(H_{\alpha}),
\end{align}
where $n_{\alpha}=\dim(H_{\alpha})$. We denote by
\begin{align}
    \operatorname{Pol}(\g)=\operatorname{span}
    \left\{ u^{\alpha}_{ij}: \alpha\in\operatorname{Irr}(\g), \ 1\leq i,j\leq n_{\alpha} \right\}
\end{align}
the {\it polynomial algebra} of $\g$. 

For any finite-dimensional unitary representation $V\in L^{\infty}(\g)\overline{\otimes}B(H_V)$, there exist a subset $E\subseteq \operatorname{Irr}(\g)$ and multiplicities $(m_{\alpha})_{\alpha\in E}\subseteq \n$ and a unitary operator $\displaystyle U:H_V \rightarrow \bigoplus_{\alpha\in E}H_{\alpha}\otimes\Comp^{m_{\alpha}}$ such that
\begin{equation}
    (1\otimes U)V(1\otimes U^*)=\bigoplus_{\alpha\in E}m_{\alpha}u^{\alpha},
\end{equation}
where $m_{\alpha}u^{\alpha}$ is the direct sum of $m_{\alpha}$ copies of $u^{\alpha}$. In this case, we write $\displaystyle V\cong\bigoplus_{\alpha\in E}m_{\alpha}u^{\alpha}$.

The \textit{Woronowicz character} $f_1:\operatorname{Pol}(\g)\rightarrow\Comp$ is the unital algebra homomorphism determined by
\begin{equation}
    Q_u=(f_1\otimes\operatorname{id})(u)
\end{equation}
for every finite-dimensional unitary representation $u$, where $Q_u$ is the canonical positive invertible operator satisfying
\begin{align}
    &\operatorname{Tr}(AQ_u)=\operatorname{Tr}(AQ_u^{-1}), \qquad A\in\operatorname{Mor}(u,u),\\
    &(S^2\otimes\operatorname{id})(u)=(1\otimes Q_u)u(1\otimes Q_u^{-1}).
\end{align}
See \cite[Proposition 1.4.4, Definition 1.7.1, and Proposition 1.7.2]{NeTu13}. Moreover, these operators $Q_u$'s satisfy
\begin{align}
    Q_{u\oplus v}&=Q_u\oplus Q_v,\\
   \label{eq212} Q_{u {\tiny \tp} v}&=Q_u\otimes Q_v;
\end{align}
see \cite[Proposition 1.4.4 and Theorem 1.4.9]{NeTu13}.

For $\alpha\in\operatorname{Irr}(\g)$, we write $Q_{\alpha}=Q_{u^{\alpha}}$ and $d_{\alpha}=\operatorname{Tr}(Q_{\alpha})=\operatorname{Tr}(Q_{\alpha}^{-1})$. We call $d_{\alpha}$ the \textit{quantum dimension} of $\alpha$. By choosing a suitable orthonormal basis $(e_j^{\alpha})_{j=1}^{n_{\alpha}}$ of $H_{\alpha}$, we may assume that $Q_{\alpha}$ is diagonal. Then the Schur orthogonality relations are given by
\begin{align}
    h\left( (u^{\alpha}_{ij})^*u^{\beta}_{kl}\right)&=\frac{\delta_{\alpha\beta}\delta_{ik}\delta_{jl}(Q_{\alpha})_{ii}^{-1}}{d_{\alpha}},\\
    h\left(u^{\alpha}_{ij}(u^{\beta}_{kl})^*\right)&=\frac{\delta_{\alpha\beta}\delta_{ik}\delta_{jl}(Q_{\alpha})_{jj}}{d_{\alpha}},
\end{align}
for all $\alpha,\beta\in\operatorname{Irr}(\g)$,
$1\leq i,j\leq n_{\alpha}$, and
$1\leq k,l\leq n_{\beta}$. Moreover,
$Q_{\alpha}=I_{H_{\alpha}}$ for all
$\alpha\in\operatorname{Irr}(\g)$ if and only if $\g$ is of Kac type.

For a finite-dimensional representation $u=\sum_{i,j=1}^{n_u} u_{ij}\otimes e_{ij}\in L^{\infty}(\g)\overline{\otimes}B(H_u)$ with respect to a fixed orthonormal basis $(e_j)_{j=1}^{n_u}$ of $H_u$, the \textit{contragredient representation} is defined by
\begin{equation}
    u^c=\sum_{i,j=1}^{n_u}u_{ij}^*\otimes e_{ij}\in L^{\infty}(\g)\overline{\otimes}B(H_u).
\end{equation}
In particular, $(u^{\alpha})^c =\displaystyle \sum_{i,j=1}^{n_{\alpha}}(u^{\alpha}_{ij})^*\otimes e^{\alpha}_{ij}\in L^{\infty}(\g)\overline{\otimes}B(H_{\alpha})$
is not unitary in general, whereas
\begin{equation}
    \left(
        1\otimes Q_{\alpha}^{\frac{1}{2}}
    \right)
    (u^{\alpha})^c
    \left(
        1\otimes Q_{\alpha}^{-\frac{1}{2}}
    \right)
\end{equation}
is unitary.

For $\alpha\in\operatorname{Irr}(\g)$, the Schur orthogonality relations imply
\begin{align}
   \label{eq213} (h\otimes\operatorname{id})\left( [(u^{\alpha})^c]^*(u^{\alpha})^c \right)
    &=\frac{n_{\alpha}}{d_{\alpha}}Q_{\alpha},
    \\
    \label{eq214}(h\otimes\operatorname{id}) \left( (u^{\alpha})^c[(u^{\alpha})^c]^*\right)
    &=    \frac{n_{\alpha}}{d_{\alpha}}Q_{\alpha}^{-1}.
\end{align}
Consequently, for a finite-dimensional direct sum representation $\displaystyle u=\bigoplus_{\alpha\in\operatorname{Irr}(\g)} m_{\alpha}u^{\alpha}$, we have
\begin{align}
\label{eq210}(h\otimes\operatorname{id})\left([u^c]^*u^c\right)
&=\bigoplus_{\alpha\in\operatorname{Irr}(\g)}\left(\frac{n_{\alpha}}{d_{\alpha}}Q_{\alpha}\right)^{\oplus m_{\alpha}},
\\
\label{eq211}(h\otimes\operatorname{id})\left(u^c[u^c]^*\right)
&=\bigoplus_{\alpha\in\operatorname{Irr}(\g)}\left(\frac{n_{\alpha}}{d_{\alpha}}Q_{\alpha}^{-1}\right)^{\oplus m_{\alpha}}.
\end{align}

Let $\overline{\alpha}\in\operatorname{Irr}(\g)$ such that $u^{\overline{\alpha}}$ is unitarily equivalent to $\left(1\otimes Q_{\alpha}^{\frac{1}{2}} \right) \left( u^{\alpha}\right)^{c} \left(1\otimes Q_{\alpha}^{-\frac{1}{2}} \right)$. Then there exists a unitary $U_{\alpha}\in B(H_{\alpha},H_{\overline{\alpha}})$ such that 
\begin{equation}
 u^{\overline{\alpha}}= \left (1\otimes U_{\alpha}Q_{\alpha}^{\frac{1}{2}}\right )\left (u^{\alpha}\right )^c \left (1\otimes Q_{\alpha}^{-\frac{1}{2}}U_{\alpha}^*\right ).
\end{equation}
For each $\alpha\in \operatorname{Irr}(\g)$, let us consider $R_{\alpha}:\Comp\rightarrow H_{\overline{\alpha}}\otimes H_{\alpha}$ given by
\begin{equation}\label{eq-intertwiner}
    R_{\alpha}(1)=\sum_{j=1}^{n_{\alpha}} U_{\alpha}Q_{\alpha}^{-\frac{1}{2}} e^{\alpha}_j\otimes e^{\alpha}_j
\end{equation}
where $(e^{\alpha}_j)_{j=1}^{n_{\alpha}}$ is an orthonormal basis of the representation space of $u^{\alpha}$. Then we have
\begin{equation}
    \operatorname{Mor}(1,u^{\overline{\alpha}}\tp u^{\alpha})=\Comp\cdot R_{\alpha}
\end{equation}
for all $\alpha\in \operatorname{Irr}(\g)$, where $1$ is the trivial representation of $\g$. This follows from the standard solution of the conjugate equations and
Frobenius reciprocity; see \cite[Section 2.2]{NeTu13}.

\subsection{Discrete quantum groups}\label{sec-discrete}

A locally compact quantum group $\g$ is called \textit{discrete} if its dual locally compact quantum group $\widehat{\g}$ is compact. Let us briefly recall the structure of a discrete quantum group $\g=\left(\ell^\infty(\g),\Delta,\varphi,\psi\right)$. The von Neumann algebra of $\g$ admits the decomposition
\begin{equation}\label{eq230}
\ell^\infty(\g)=\ell^{\infty}\text{-}\bigoplus_{\alpha\in\operatorname{Irr}(\widehat{\g})} B(H_\alpha).
\end{equation}
The left and the right Haar weights on $\g$ are $\varphi,\psi:\ell^{\infty}(\g)_+\rightarrow [0,\infty]$ given by
\begin{align}
    \varphi(X)&=\sum_{\alpha\in \operatorname{Irr}(\widehat{\g})}d_{\alpha}\operatorname{Tr}(X_{\alpha}Q_{\alpha})\\
    \psi(X) &= \sum_{\alpha\in\operatorname{Irr}(\widehat{\g})} d_\alpha \operatorname{Tr}(X_\alpha Q_\alpha^{-1})
\end{align}
for all $X=(X_{\alpha})_{\alpha\in \operatorname{Irr}(\widehat{\g})}\in \ell^{\infty}(\g)_+$. These are normal semifinite faithful weights on $\ell^\infty(\g)$. In particular, $\varphi=\psi$ if and only if $Q_\alpha=\operatorname{Id}_{H_\alpha}$ for all $\alpha\in\operatorname{Irr}(\widehat{\g})$, or equivalently, if and only if the compact dual $\widehat{\g}$ is of Kac type. In this case, $\g$ is said to be of Kac type.

For each $\alpha\in\operatorname{Irr}(\widehat{\g})$, let $\rho_{\alpha}:\ell^{\infty}(\g)\rightarrow B(H_{\alpha})$ be the
canonical coordinate map. As explained in \cite[Definition 1.6.8 and pp. 28--29]{NeTu13}, the comultiplication on $\ell^{\infty}(\g)$ can be described explicitly in terms of the representation category of the compact dual $\widehat{\g}$. For fixed $\alpha,\beta\in\operatorname{Irr}(\widehat{\g})$, let $m_{\alpha,\beta}^{\gamma}$ denote the multiplicity of $u^{\gamma}$ in $u^{\alpha}\tp u^{\beta}$, and choose isometries $T_{\gamma,r}^{\alpha,\beta}\in \operatorname{Mor}(u^{\gamma},u^{\alpha}\tp u^{\beta})$, $1\leq r\leq m_{\alpha,\beta}^{\gamma}$, whose ranges are mutually orthogonal and satisfy
\begin{equation}
    \sum_{\gamma\in\operatorname{Irr}(\widehat{\g})}\sum_{r=1}^{m_{\alpha,\beta}^{\gamma}}
    T_{\gamma,r}^{\alpha,\beta}\left(T_{\gamma,r}^{\alpha,\beta}\right)^*
    =\operatorname{Id}_{H_{\alpha}}\otimes \operatorname{Id}_{H_{\beta}}.
\end{equation}
Then the $(\alpha,\beta)$-block of the comultiplication is given by
\begin{equation}
    (\rho_{\alpha}\otimes\rho_{\beta})(\Delta(X)) = \sum_{\gamma\in\operatorname{Irr}(\widehat{\g})} \sum_{r=1}^{m_{\alpha,\beta}^{\gamma}} T_{\gamma,r}^{\alpha,\beta} \rho_{\gamma}(X) \left(T_{\gamma,r}^{\alpha,\beta}\right)^*, \qquad X\in\ell^{\infty}(\g).
\end{equation}
For each fixed pair $(\alpha,\beta)$, the sum above is finite. Equivalently, for any $T\in\operatorname{Mor}(u^{\gamma},u^{\alpha}\tp u^{\beta})$,
\begin{equation}\label{eq200}
    (\rho_{\alpha}\otimes\rho_{\beta})(\Delta(X))T=T\rho_{\gamma}(X), \qquad X\in\ell^{\infty}(\g).
\end{equation}

Let $\epsilon\in\operatorname{Irr}(\widehat{\g})$ denote the equivalence class
of the trivial representation. Since $H_\epsilon=\Comp$, we identify
$B(H_\epsilon)$ with $\Comp$. The coordinate map $\rho_\epsilon\in \ell^1(\g)=\ell^\infty(\g)_*$ is a normal unital $*$-homomorphism, which we call the \textit{counit}. Then $\rho_{\epsilon}$ satisfies
\begin{equation}
    (\rho_{\epsilon}\otimes\operatorname{id})\Delta=\operatorname{id}=(\operatorname{id}\otimes\rho_{\epsilon})\Delta.
\end{equation}
Equivalently, $\rho_{\epsilon}$ is the unit in the convolution algebra $\ell^1(\g)$, i.e. $\rho_{\epsilon}\star\omega=\omega=\omega\star\rho_{\epsilon}$ for all $\omega\in \ell^1(\g)$. 

For a representation $V\in\ell^{\infty}(\g)\overline{\otimes}B(H_{V})$ and
$\alpha\in\operatorname{Irr}(\widehat{\g})$, we define the $\alpha$-block of $V$ by
\begin{equation}
    V^{\alpha}=(\rho_{\alpha}\otimes\operatorname{id})(V)\in B(H_{\alpha})\overline{\otimes}B(H_{V}).
\end{equation}
Under the decomposition \eqref{eq230}, the representation $V$ is identified with the bounded family
$(V^{\alpha})_{\alpha\in\operatorname{Irr}(\widehat{\g})}$, and $\displaystyle \norm{V}=\sup_{\alpha\in\operatorname{Irr}(\widehat{\g})}\norm{V^{\alpha}}$. Furthermore, \eqref{eq200} implies that for any $T\in\operatorname{Mor}(u^{\gamma},u^{\alpha}\tp u^{\beta})$,
\begin{align}
  \notag & (\rho_{\alpha}\otimes\rho_{\beta}\otimes \operatorname{id})(V_{13}V_{23})(T\otimes \operatorname{Id}_{H_V})\\
  \label{eq200.5} &=(\rho_{\alpha}\otimes\rho_{\beta}\otimes \operatorname{id})((\Delta\otimes \id)(V))(T\otimes \operatorname{Id}_{H_V})=(T\otimes \operatorname{Id}_{H_V})(\rho_{\gamma}\otimes \operatorname{id})(V)
\end{align}

\begin{lemma}\label{lem-intertwiner}
   Let $V\in \ell^{\infty}(\g)\overline{\otimes}B(H_V)$ be a non-degenerate representation of a discrete quantum group $\g$. Then, for any $\alpha\in \operatorname{Irr}(\widehat{\g})$ and $T\in \text{Mor}(u^{\epsilon},u^{\overline{\alpha}}\tp u^{\alpha})$, we have
\begin{equation}\label{eq202}
V^{\overline{\alpha}}_{13}V^\alpha_{23}(T\otimes 1_{H_V})=T\otimes 1_{H_V}.
\end{equation}
\end{lemma}

\begin{proof}

Note that, by \eqref{eq200.5}, we have
\begin{align}
V^{\overline{\alpha}}_{13}V^\alpha_{23}(T \otimes 1_{H_V})&=(\rho_{\overline{\alpha}}\otimes \rho_{\alpha}\otimes \operatorname{id})\left ((\Delta\otimes \operatorname{id}) (V)\right )(T\otimes 1_{H_V})\\
&=(T\otimes 1_{H_V}) (\rho_{\epsilon}\otimes \operatorname{id})(V) =(T\otimes 1_{H_V})V^\epsilon.
\end{align}

Thus, the remaining part is to show $V^{\epsilon}=\text{Id}_{H_V}$. By \cite{BrDaSa13}, a representation $V\in \ell^{\infty}(\g)\overline{\otimes}B(H_V)$ defines a completely bounded algebra homomorphism $\pi_V:\ell^1(\g)\rightarrow B(H_{V})$ given by $\pi_{V}(\omega) =(\omega\otimes\operatorname{id})(V)$ for all $\omega\in \ell^1(\g)$. Then $P=\pi_V(\rho_{\epsilon})=(\rho_{\epsilon}\otimes \operatorname{id})(V)=V^{\epsilon}$ satisfies $P^2=P$ in $B(H_V)$. Since $\pi_V(\omega)\xi=\pi_V(\rho_{\epsilon}\star \omega)\xi=P\cdot \pi_V(\omega)\xi$ we have
\begin{align}
    \left\{\pi_V(\omega)\xi: \omega\in \ell^1(\g),\xi\in H_V\right\}\subseteq \left\{P\xi':\xi'\in H_V\right\}=\text{ran}(P).
\end{align}
Thus, the non-degeneracy of $V$ implies $P=\pi_V(\rho_\epsilon)=V^{\epsilon}=\operatorname{Id}_{H_V}$.
    
\end{proof}

\section{The main results}
\subsection{Contractive representations are unitary} \label{sec:contractive-representations}

In this section, we show that all non-degenerate contractive representations of compact quantum groups and discrete quantum groups are automatically unitary. Let us begin with the case of compact quantum groups.

\begin{theorem}
\label{thm-main-1}
Let $\g=(L^\infty(\g),\Delta,h)$ be a compact quantum group and let $V\in L^\infty(\g)\overline{\otimes}B(H_V)$ be a non-degenerate representation satisfying $\norm{V}_{L^\infty(\g)\overline{\otimes}B(H_V)}\leq 1$. Then $V$ is unitary.
\end{theorem}

\begin{proof}
Since $\norm{V}_{L^\infty(\g)\overline{\otimes}B(H_V)}\leq 1$, the operator
\begin{equation}
D=1-V^*V
\in L^\infty(\g)\overline{\otimes}B(H_V)
\end{equation}
is positive. Since $(\Delta\otimes\operatorname{id})(V)=V_{13}V_{23}$ and $\Delta\otimes\operatorname{id}$ is a unital $*$-homomorphism, we have
\begin{align}
(\Delta\otimes\operatorname{id})(D)&=
1-(\Delta\otimes\operatorname{id})(V)^*(\Delta\otimes\operatorname{id})(V)=
1-V_{23}^*V_{13}^*V_{13}V_{23}
\\
\label{eq302}&= 1-V_{23}^*V_{23}+V_{23}^*\left(1-V_{13}^*V_{13}\right)V_{23}=D_{23}+V_{23}^*D_{13}V_{23}.
\end{align}

Note that the right invariance of the Haar state $h$ implies
\begin{align}
(\operatorname{id}\otimes h\otimes\operatorname{id}) \left((\Delta\otimes\operatorname{id})(D)\right) & = \left((\operatorname{id}\otimes h)\Delta\otimes\operatorname{id}\right )(D)\notag \\
\label{eq400}=& 1\otimes (h\otimes \operatorname{id})(D)=(\operatorname{id}\otimes h \otimes \operatorname{id})(D_{23}).
\end{align}
Then, by \eqref{eq302} and \eqref{eq400}, we obtain
\begin{equation}
(\operatorname{id}\otimes h\otimes\operatorname{id}) (D_{23}+V_{23}^*D_{13}V_{23}) = (\operatorname{id}\otimes h\otimes\operatorname{id})(D_{23}),
\end{equation}
which implies
\begin{equation}\label{eq303}
    (\operatorname{id}\otimes h\otimes\operatorname{id}) (V_{23}^*D_{13}V_{23})=0.
\end{equation}
Since $\operatorname{id}\otimes h\otimes\operatorname{id}$ is faithful, \eqref{eq303} implies
\begin{equation}
\left (D_{13}^{\frac{1}{2}}V_{23}\right )^*D_{13}^{\frac{1}{2}}V_{23}=V_{23}^*D_{13}V_{23}=0
\end{equation}
or equivalently, $D_{13}^{\frac{1}{2}}V_{23}=0$. Thus, for any $\phi\in L^\infty(\g)_*$, we obtain
\begin{equation}
(\operatorname{id}\otimes \phi\otimes \operatorname{id})\left (D_{13}^{\frac{1}{2}}V_{23}\right ) = D^{\frac{1}{2}}\left[ 1\otimes (\phi\otimes\operatorname{id})(V) \right] = 0.
\label{eq304}
\end{equation}

Since the non-degeneracy of $V$ means that
\begin{equation}
\overline{\operatorname{span} \left\{(\phi\otimes\operatorname{id})(V)\xi: \phi\in L^\infty(\g)_*, \xi\in H_V
\right\}} =H_V,
\end{equation}
it follows from \eqref{eq304} that $D^{1/2}$
vanishes on a dense subspace of $L^2(\g)\otimes H_V$. Hence we can conclude that $D=1-V^*V=0$, i.e. $V^*V=1$. Thus, $V$ is an isometric representation.

Lastly, every isometric representation of a locally compact quantum group is
unitary by \cite[Corollary 4.16]{BrDaSa13}. Thus, $V$ is unitary as desired.
\end{proof}

We now prove the corresponding result for discrete quantum groups.

\begin{theorem}\label{thm:discrete-contractive-representation-unitary}
Let $\g$ be a discrete quantum group, and let $V\in \ell^\infty(\g)\overline{\otimes}B(H_V)$ be a non-degenerate representation satisfying $\norm{V}_{\ell^\infty(\g)\overline{\otimes}B(H_V)} \leq 1$. Then $V$ is unitary.
\end{theorem}

\begin{proof}
Write $V=\displaystyle \sum_{\alpha\in \operatorname{Irr}(\widehat{\g})}V^{\alpha}$ with $V^{\alpha} = (\rho_{\alpha}\otimes\operatorname{id})(V) \in B(H_{\alpha}) \overline{\otimes} B(H_{V})$. Since $\rho_{\alpha}$ is a $*$-homomorphism, we have $\norm{V^{\alpha}}\leq \norm{V}\leq 1$. Let us consider 
\begin{equation}\label{eq305}
    D_{\alpha}= 1_{H_{\alpha}\otimes H_{V}}-(V^{\alpha})^{*}V^{\alpha}\geq 0.
\end{equation}
To obtain the desired conclusion, it is enough to show that $D_{\alpha}=0$ for all $\alpha\in \operatorname{Irr}(\widehat{\g})$.

By Lemma \ref{lem-intertwiner} and \eqref{eq-intertwiner},
\begin{equation}
    V_{13}^{\overline{\alpha}}V_{23}^{\alpha}(R_{\alpha}\otimes 1_{H_{V}})=R_{\alpha}\otimes 1_{H_{V}},
\end{equation}
where $\displaystyle R_{\alpha}=\sum_{j=1}^{n_{\alpha}} U_{\alpha}Q_{\alpha}^{-\frac{1}{2}} e^{\alpha}_j\otimes e^{\alpha}_j\in \operatorname{Mor}(u^{\epsilon},u^{\overline{\alpha}}\tp u^{\alpha})$. Therefore
\begin{align}
\norm{(R_{\alpha}\otimes 1_{H_{V}})\xi}=\norm{V_{13}^{\overline{\alpha}}V_{23}^{\alpha}(R_{\alpha}\otimes 1_{H_{V}})\xi} \leq \norm{V_{23}^{\alpha}(R_{\alpha}\otimes 1_{H_{V}})\xi}\leq \norm{(R_{\alpha}\otimes 1_{H_{V}})\xi}
\end{align}
for all $\xi\in H_V$. Thus, we obtain
\begin{align}
    \norm{V_{13}^{\overline{\alpha}}V_{23}^{\alpha}(R_{\alpha}\otimes 1_{H_{V}})\xi} = \norm{V_{23}^{\alpha}(R_{\alpha}\otimes 1_{H_{V}})\xi}= \norm{(R_{\alpha}\otimes 1_{H_{V}})\xi}
\end{align}
for all $\xi\in H_V$. This implies that
\begin{align}
\norm{(1_{H_{\overline{\alpha}}}\otimes D_{\alpha}^{\frac{1}{2}})(R_{\alpha}\otimes 1_{H_V})\xi}^2&=\la (1_{H_{\overline{\alpha}}}\otimes D_{\alpha}^{\frac{1}{2}})(R_{\alpha}\otimes 1_{H_V})\xi, (1_{H_{\overline{\alpha}}}\otimes D_{\alpha}^{\frac{1}{2}}) (R_{\alpha}\otimes 1_{H_V}) \xi\ra\\
&=\la (1_{H_{\overline{\alpha}}}\otimes D_{\alpha})(R_{\alpha}\otimes 1_{H_V})\xi, (R_{\alpha}\otimes 1_{H_V}) \xi \ra\\
&=\norm{(R_{\alpha}\otimes 1_{H_V})\xi}^2-\norm{V^{\alpha}_{23}(R_{\alpha}\otimes 1_{H_V})\xi}^2=0.
\end{align}
The second last equality is thanks to \eqref{eq305}. Thus, it follows that
\begin{align}
 0=\left(1_{H_{\overline{\alpha}}} \otimes D_{\alpha}^{\frac{1}{2}}\right) (R_{\alpha}\otimes 1_{H_{V}}) \xi=\sum_{j=1}^{n_{\alpha}} U_{\alpha}Q_{\alpha}^{-\frac{1}{2}} e^{\alpha}_j\otimes D_{\alpha}^{\frac{1}{2}}(e^{\alpha}_j\otimes \xi),
\end{align}
which is equivalent to $D_{\alpha}^{\frac{1}{2}}(e^{\alpha}_j\otimes \xi)=0$ for all $1\leq j\leq n_{\alpha}$ and $\xi\in H_V$. This means that $D_{\alpha}^{\frac{1}{2}}=0$ on $H_{\alpha}\otimes H_V$, so we can conclude that $D_{\alpha}=0$ for all $\alpha\in \operatorname{Irr}(\widehat{\g})$ and  $V$ is an isometric representation. Lastly, every isometric representation of a locally compact quantum group is unitary by \cite[Corollary 4.16]{BrDaSa13}. Therefore, $V$ is unitary.
\end{proof}

\begin{corollary}
Let $\g$ be an arbitrary compact or discrete quantum group. Then any non-degenerate completely contractive algebra homomorphism $\pi:L^1(\g)\longrightarrow B(H_\pi)$ is automatically a $*$-homomorphism.
\end{corollary}

\begin{proof}
Under the standard completely isometric correspondence between completely bounded maps on $L^1(\g)$ and elements of $L^\infty(\g)\overline{\otimes}B(H_\pi)$, the completely contractive algebra homomorphism $\pi$ is associated with a contractive representation $V_\pi\in L^\infty(\g)\overline{\otimes}B(H_\pi)$ satisfying
\begin{equation}
\pi(\varphi) = (\varphi\otimes\operatorname{id})(V_\pi), \qquad \varphi\in L^1(\g).
\end{equation}
The non-degeneracy of $\pi$ is equivalent to the non-degeneracy of $V_{\pi}$. Hence, by Theorem \ref{thm-main-1} and Theorem \ref{thm:discrete-contractive-representation-unitary}, $V_\pi$ is unitary. Hence $\pi$ is a $*$-homomorphism; see
\cite{BrDaSa13}.
\end{proof}

\subsection{Non-unitarizable non-degenerate representations of compact quantum groups}\label{sec-non-unitarizability}
In \cite{BrYo19}, the authors established the existence of non-unitarizable non-degenerate representations under the assumption that the dual discrete quantum group has subexponential growth. They also constructed a non-unitarizable non-degenerate representation for the somewhat artificial example
$\g=\prod_{n=1}^{\infty}SU_{q_n}(2)$, where $(q_n)_{n=1}^{\infty}\subseteq(0,1)$ satisfies $\displaystyle \lim_{n\rightarrow\infty}q_n=0$. However, for individual quantum groups such as $\g=SU_q(2)$, their argument does not provide an explicit non-unitarizable non-degenerate representation. The main obstruction is that the norms of the contragredient representations are not uniformly bounded in general. Actually, the norms are unbounded if $\g$ is non-Kac $O_F^+$ or non-Kac $U_F^+$ by the following proposition.

\begin{proposition}\label{prop300}
For any compact quantum group $\g$ and $\alpha\in \operatorname{Irr}(\g)$,
\begin{align}
\label{eq30}
\norm{\left(u^{\alpha}\right)^c}_{L^{\infty}(\g)\overline{\otimes}M_{n_{\alpha}}}^2 \geq \frac{n_{\alpha}}{d_{\alpha}}
\max\left\{\norm{Q_{\alpha}}_{\mathrm{op}},\norm{Q_{\alpha}^{-1}}_{\mathrm{op}}\right\}.
\end{align}
In particular, if $\g$ is either a non-Kac free orthogonal quantum group $O_F^+$ or a non-Kac free unitary quantum group $U_F^+$, then
\begin{align}
\sup_{\alpha\in\operatorname{Irr}(\g)}\norm{\left(u^{\alpha}\right)^c}_{L^{\infty}(\g)\overline{\otimes}M_{n_{\alpha}}}
=\infty.
\end{align}

\end{proposition}

\begin{proof}
First of all, by \eqref{eq213} and \eqref{eq214}, 
\begin{align}
   (h\otimes \operatorname{id})\left ( \left ( u^{\alpha}\right )^c\left [ \left ( u^{\alpha}\right )^c \right ]^*\right )&=\frac{n_{\alpha}}{d_{\alpha}}Q_{\alpha}^{-1},\\
   (h\otimes \operatorname{id})\left ( \left [ \left ( u^{\alpha}\right )^c \right ]^* \left ( u^{\alpha}\right )^c \right )&=\frac{n_{\alpha}}{d_{\alpha}}Q_{\alpha}.
\end{align}
This implies
\begin{align}
    &\norm{\left ( u^{\alpha}\right )^c}_{L^{\infty}(\g)\overline{\otimes}M_{n_{\alpha}}}^2\\
    &\geq \max\left\{\norm{(h\otimes \operatorname{id})\left ( \left ( u^{\alpha}\right )^c\left [ \left ( u^{\alpha}\right )^c \right ]^*\right )}_{M_{n_{\alpha}}},\norm{(h\otimes \operatorname{id})\left (\left [ \left ( u^{\alpha}\right )^c \right ]^* \left ( u^{\alpha}\right )^c\right )}_{M_{n_{\alpha}}}\right\}\\
    &\geq \frac{n_{\alpha}}{d_{\alpha}}\max\left\{\norm{Q_{\alpha}^{-1}}_{\text{op}},\norm{Q_{\alpha}}_{\text{op}}
\right\}
\end{align}
since the Haar state $h$ is completely contractive on $L^{\infty}(\g)$.

From now on, let us verify the last assertion for free orthogonal and free unitary quantum groups. First, let $\g=O_F^+$ be of non-Kac type. Recall that $\operatorname{Irr}(O_F^+)$ is indexed by $\mathbb{N}_0$ by \cite{Ba96}, and
\begin{equation}
\norm{Q_k}_{\mathrm{op}}=\norm{Q_k^{-1}}_{\mathrm{op}}=\norm{Q_1}_{\mathrm{op}}^k.
\end{equation}
When $n_1=2$, we have $n_k=k+1$ and $\displaystyle \frac{\norm{Q_k}_{\mathrm{op}}}{d_k}=\frac{\norm{Q_k^{-1}}_{\mathrm{op}}}{d_k}=\frac{1-q^2}{1-q^{2k+2}}\in (1-q^2,1]$. Consequently,
\begin{equation}
\frac{n_k}{d_k}\max\left\{\norm{Q_k}_{\mathrm{op}},\norm{Q_k^{-1}}_{\mathrm{op}}\right\}\geq n_k(1-q^2)\longrightarrow_{k}\infty.
\end{equation}

Suppose next that $n_1\geq 3$. It was shown in the proof of \cite[Theorem 3.3]{BrVeYo19} that $\displaystyle \frac{n_k\norm{Q_1}_{\mathrm{op}}^k}{d_k}$ grows exponentially as $k\to\infty$. Since $\norm{Q_k}_{\mathrm{op}}=\norm{Q_1}_{\mathrm{op}}^k$, we obtain
\begin{equation}
\frac{n_k}{d_k}\max\left\{\norm{Q_k}_{\mathrm{op}},\norm{Q_k^{-1}}_{\mathrm{op}}\right\}=\frac{n_k\norm{Q_1}_{\mathrm{op}}^k}{d_k}\longrightarrow_{k} \infty.
\end{equation}
Thus, \eqref{eq30} implies
\begin{equation}
\sup_{\alpha\in\operatorname{Irr}(O_F^+)}
\norm{\left(u^{\alpha}\right)^c}_{L^{\infty}(O_F^+)\overline{\otimes}M_{n_{\alpha}}}
=\infty.
\end{equation}

Finally, let $\g=U_F^+$ be a non-Kac free unitary quantum group with $F\in\operatorname{GL}_N(\Comp)$. The irreducible representations of $U_F^+$ are indexed by the free monoid generated by two letters $a$ and $b=\overline{a}$, where the fundamental representation corresponds to $a$. By the fusion rules of $U_F^+$ \cite[Theorem 1]{Ban97}, $u^{a^k}=(u^a)^{{\tiny\tp} k}$ is irreducible for all $k\in \n_0$, and $n_{a^k}=n_a^k$, $d_{a^k}=d_a^k$, $Q_{a^k}=Q_a^{\otimes k}$. Since $U_F^+$ is of non-Kac type, it follows that $\displaystyle \frac{n_a\norm{Q_a}_{\mathrm{op}}}{d_a}>1$. This implies
\begin{align}
\frac{n_{a^k}}{d_{a^k}}\norm{Q_{a^k}}_{\mathrm{op}}=\left(\frac{n_a\norm{Q_a}_{\mathrm{op}}}{d_a}\right)^k \longrightarrow_{k}\infty.
\end{align}
Therefore, \eqref{eq30} implies
\begin{equation}
\sup_{\alpha\in\operatorname{Irr}(U_F^+)}
\norm{\left(u^{\alpha}\right)^c}_{L^{\infty}(U_F^+)\overline{\otimes}M_{n_{\alpha}}}
=\infty.
\end{equation}
\end{proof}
To overcome the unboundedness issue raised in Proposition \ref{prop300}, our strategy is to consider intermediate representations between the contragredient representation $(u^{\alpha})^c$ and its unitarization
\begin{equation}
(1\otimes Q_{\alpha}^{\frac{1}{2}})(u^{\alpha})^c(1\otimes Q_{\alpha}^{-\frac{1}{2}}).    
\end{equation}
More precisely, for $0\leq \theta\leq 1$, we consider
\begin{equation}
v^{\alpha}_{\theta}=(1\otimes Q_{\alpha}^{\frac{\theta}{2}})(u^{\alpha})^c
(1\otimes Q_{\alpha}^{-\frac{\theta}{2}}).
\end{equation}
Vector-valued complex interpolation plays a crucial role in controlling the norms of these intermediate representations. A general argument is as follows.

\begin{lemma}\label{lem30}
For any $V\in L^{\infty}(\g)\overline{\otimes} M_n$, an invertible positive $T\in M_n$, and $\theta\in (0,1)$, we have
\begin{align}
    &\norm{\left (1\otimes T^{\theta}\right )V\left (1\otimes T^{-\theta}\right )}_{L^{\infty}(\g)\overline{\otimes} M_n}\leq \norm{V}_{L^{\infty}(\g)\overline{\otimes} M_n}^{1-\theta}\cdot \norm{\left (1\otimes T\right )V\left (1\otimes T^{-1}\right )}_{L^{\infty}(\g)\overline{\otimes} M_n}^{\theta}.
\end{align}
\end{lemma}

\begin{proof}
Consider the $L^{\infty}(\g)\overline{\otimes}M_n$-valued analytic
function
\begin{equation}
F(z)=(1\otimes T^z)V(1\otimes T^{-z}),
\qquad 0\leq\operatorname{Re}(z)\leq 1.
\end{equation}
Since $T^{it}$ is unitary for any $t\in\mathbb{R}$, we have
\begin{equation}
\norm{F(it)}_{L^{\infty}(\g)\overline{\otimes}M_n}=\norm{V}_{L^{\infty}(\g)\overline{\otimes}M_n}.
\end{equation}
Similarly,
\begin{align}
\norm{F(1+it)}_{L^{\infty}(\g)\overline{\otimes}M_n}&=\norm{(1\otimes T^{it})(1\otimes T)V(1\otimes T^{-1})(1\otimes T^{-it})}_{L^{\infty}(\g)\overline{\otimes}M_n}
\\
&=
\norm{(1\otimes T)V(1\otimes T^{-1})}_{L^{\infty}(\g)\overline{\otimes}M_n}.
\end{align}
Thus, the Hadamard three-line theorem for Banach space-valued
analytic functions \cite[Appendix to Section IX.4]{ReSi75} gives
\begin{align}
&\norm{(1\otimes T^{\theta})V(1\otimes T^{-\theta})}_{L^{\infty}(\g)\overline{\otimes}M_n}
=\norm{F(\theta)}_{L^{\infty}(\g)\overline{\otimes}M_n}\\
&\leq \norm{V}_{L^{\infty}(\g)\overline{\otimes}M_n}^{1-\theta} \norm{(1\otimes T)V(1\otimes T^{-1})}_{L^{\infty}(\g)\overline{\otimes}M_n}^{\theta}.
\end{align}
\end{proof}

Applying Lemma \ref{lem30} gives the following estimate.

\begin{proposition}\label{prop30}
Let $\g$ be a compact quantum group, $\alpha\in \operatorname{Irr}(\g)$ and $\theta(\alpha)\in [0,1]$. Then the following representation
\begin{equation}
    v^{\alpha}_{\theta(\alpha)}=\left (1\otimes Q_{\alpha}^{\frac{\theta(\alpha)}{2}}\right )(u^{\alpha})^c \left (1\otimes Q_{\alpha}^{-\frac{\theta(\alpha)}{2}}\right )\in L^{\infty}(\g)\otimes M_{n_{\alpha}}
\end{equation}
satisfies $\norm{v^{\alpha}_{\theta(\alpha)}}_{L^{\infty}(\g)\overline{\otimes}M_{n_{\alpha}}}\leq n_{\alpha}^{1-\theta(\alpha)}$.     
\end{proposition}

\begin{proof}
For any representation $u\in L^{\infty}(\g)\overline{\otimes}B(H_u)$, it is straightforward to see that the following conjugate
\begin{equation}
u_T= (1\otimes T)u(1\otimes T^{-1})\in L^{\infty}(\g)\overline{\otimes}B(H_u)
\end{equation}
is also a representation for any invertible $T\in B(H_u)$. Thus, the conjugate 
\begin{equation}
    v^{\alpha}_{\theta(\alpha)}=\left (1\otimes Q_{\alpha}^{\frac{\theta(\alpha)}{2}}\right )(u^{\alpha})^c \left (1\otimes Q_{\alpha}^{-\frac{\theta(\alpha)}{2}}\right )
\end{equation}
is also a representation of $\g$.

 The contragredient representation $v^{\alpha}_0=(u^{\alpha})^c$ satisfies $\norm{v^{\alpha}_0}_{L^{\infty}(\g)\overline{\otimes}M_{n_{\alpha}}}\leq n_{\alpha}$ since the completely bounded norm of the transpose map on $M_{n_{\alpha}}$ is $n_{\alpha}$ \cite{Tom83}. On the other hand, its unitarized version $v^{\alpha}_1=\left (1\otimes Q_{\alpha}^{\frac{1}{2}}\right )\left (u^{\alpha}\right )^c \left (1\otimes Q_{\alpha}^{-\frac{1}{2}}\right )$ satisfies $\norm{v^{\alpha}_1}_{L^{\infty}(\g)\overline{\otimes}M_{n_{\alpha}}}=1$. Hence, by Lemma \ref{lem30}, we obtain
 \begin{equation}
     \norm{v^{\alpha}_{\theta(\alpha)}}_{L^{\infty}(\g)\overline{\otimes}M_{n_{\alpha}}}\leq \norm{v^{\alpha}_0}_{L^{\infty}(\g)\overline{\otimes}M_{n_{\alpha}}}^{1-\theta(\alpha)}\norm{v^{\alpha}_1}_{L^{\infty}(\g)\overline{\otimes}M_{n_{\alpha}}}^{\theta(\alpha)}\leq  n_{\alpha}^{1-\theta(\alpha)} .
 \end{equation}

\end{proof}

From now on, let us fix a subset $E\subseteq \operatorname{Irr}(\g)$ and $H_E=\displaystyle \bigoplus_{\alpha\in E}H_{\alpha}$, and consider functions $\theta:E \rightarrow [0,1]$. Suppose that there exists a constant $C>0$ satisfying
\begin{equation}
\sup_{\alpha\in E} n_{\alpha}^{1-\theta(\alpha)}\leq C.
\end{equation}
Then the following non-degenerate direct sum representation
\begin{align}
    \label{eq02}&v_{E,\theta}=\bigoplus_{\alpha\in E}v^{\alpha}_{\theta(\alpha)} = \sum_{\alpha\in E}\sum_{i,j=1}^{n_{\alpha}}\left [v^{\alpha}_{\theta(\alpha)}\right ]_{ij}\otimes e^{\alpha}_{ij}\\
    &=\sum_{\alpha\in E}\sum_{i,j=1}^{n_{\alpha}}(Q_{\alpha})_{ii}^{\frac{\theta(\alpha)}{2}}(Q_{\alpha})_{jj}^{-\frac{\theta(\alpha)}{2}}(u^{\alpha}_{ij})^*\otimes e^{\alpha}_{ij} \in L^{\infty}(\g)\overline{\otimes} B(H_E)
\end{align}
satisfies $\displaystyle \norm{v_{E,\theta}}_{L^{\infty}(\g)\overline{\otimes} B(H_E)}\leq C$ by Proposition \ref{prop30}. The following theorem establishes a useful sufficient and necessary condition for the representation $v_{E,\theta}$ to be unitarizable.

\begin{theorem}\label{thm2}
Let $\g$ be a compact quantum group and $E$ be a subset of $\operatorname{Irr}(\g)$. Suppose that a function $\theta:E\rightarrow [0,1]$ satisfies $\displaystyle \sup_{\alpha\in E}n_{\alpha}^{1-\theta(\alpha)}<\infty$. Then the non-degenerate representation $v_{E,\theta}$ in \eqref{eq02} is unitarizable if and only if $\displaystyle \sup_{\alpha\in E}\left ( \norm{Q_{\alpha}}_{op}\norm{Q_{\alpha}^{-1}}_{op} \right )^{\frac{1-\theta(\alpha)}{2}}<\infty$.
\end{theorem}
\begin{proof}
We first suppose that $\displaystyle \sup_{\alpha\in E}\left ( \norm{Q_{\alpha}}_{op}\norm{Q_{\alpha}^{-1}}_{op} \right )^{\frac{1-\theta(\alpha)}{2}}<\infty$.  Since $\norm{Q_{\alpha}}_{op}\geq 1$ and $\norm{Q_{\alpha}^{-1}}_{op}\geq 1$, the assumption implies $\displaystyle \sup_{\alpha\in E} \norm{Q_{\alpha}}_{op}^{\frac{1-\theta(\alpha)}{2}} <\infty$ and $\displaystyle \sup_{\alpha\in E} \norm{Q_{\alpha}^{-1}}_{op}^{\frac{1-\theta(\alpha)}{2}} <\infty$. Thus, $\displaystyle T = \bigoplus_{\alpha\in E} Q_{\alpha}^{\frac{1-\theta(\alpha)}{2}}$ is bounded and invertible in $B(H_E)$. Moreover,
\begin{align}
(1\otimes T)v_{E,\theta}(1\otimes T^{-1})&=(1\otimes T)\left(\bigoplus_{\alpha\in E}(1\otimes Q_{\alpha}^{\frac{\theta(\alpha)}{2}})(u^\alpha)^c(1\otimes Q_{\alpha}^{-\frac{\theta(\alpha)}{2}})
\right)(1\otimes T^{-1}) \\
&= \bigoplus_{\alpha\in E}(1\otimes Q_{\alpha}^{\frac12})(u^\alpha)^c(1\otimes Q_{\alpha}^{-\frac12})
\end{align}
is unitary, and hence $v_{E,\theta}$ is unitarizable.

It remains to prove the converse. Suppose that $v_{E,\theta}$ is unitarizable. Then there exists an invertible operator $T\in B(H_E)$ such that $(1\otimes T)v_{E,\theta}(1\otimes T^{-1})$ is unitary. We will show that this implies
\begin{equation}
\sup_{\alpha\in E}\left(\norm{Q_{\alpha}}_{op}\norm{Q_{\alpha}^{-1}}_{op}\right)^{\frac{1-\theta(\alpha)}{2}}\leq \norm{T}_{B(H_E)}\norm{T^{-1}}_{B(H_E)}<\infty.
\end{equation}

Since $(1\otimes T)v_{E,\theta}(1\otimes T^{-1})$ is a unitary representation whose  irreducible decomposition is identified with $\left\{\overline{\alpha}: \alpha\in E\right\}$, each with multiplicity $1$, there exists a unitary operator $U\in B(H_E)$ such that
\begin{align}
(1\otimes UT)v_{E,\theta}(1\otimes T^{-1}U^{-1})&=(1\otimes U)(1\otimes T)v_{E,\theta}(1\otimes T^{-1})(1\otimes U^{-1})\\
&=\bigoplus_{\alpha\in E}(1\otimes Q_{\alpha}^{\frac{1}{2}})(u^{\alpha})^c (1\otimes Q_{\alpha}^{-\frac{1}{2}}).
\end{align}
In other words, we have
\begin{align}
&\sum_{\alpha\in E}\sum_{i,j=1}^{n_{\alpha}}(Q_{\alpha})_{ii}^{\frac{\theta(\alpha)}{2}}(Q_{\alpha})_{jj}^{-\frac{\theta(\alpha)}{2}}(u^{\alpha}_{ij})^*\otimes UTe^{\alpha}_{ij}T^{-1}U^{-1}\\
&=\sum_{\alpha\in E}\sum_{i,j=1}^{n_{\alpha}}(Q_{\alpha})_{ii}^{\frac{1}{2}}(Q_{\alpha})_{jj}^{-\frac{1}{2}}(u^{\alpha}_{ij})^*\otimes e^{\alpha}_{ij}.
\end{align}
Thanks to the linear independence of the matrix elements $(u^{\alpha}_{ij})^*$, we obtain
\begin{align}
    UTe^{\alpha}_{ij}T^{-1}U^{-1}= (Q_{\alpha})_{ii}^{\frac{1-\theta(\alpha)}{2}}(Q_{\alpha})_{jj}^{\frac{\theta(\alpha)-1}{2}}e^{\alpha}_{ij}.
\end{align}
This implies $\displaystyle UT=\bigoplus_{\alpha\in E}(UT)(\alpha)\in B(H_E)$ and there exist scalars $x(\alpha)\in \Comp\setminus \left\{0\right\}$ such that
\begin{align}
Q_{\alpha}^{\frac{\theta(\alpha)-1}{2}}(UT)(\alpha)=x(\alpha)\operatorname{Id}_{\alpha}
\end{align}
for all $\alpha\in E$. Thus, we obtain
\begin{align}
    &\norm{T}_{B(H_E)}=\norm{UT}_{B(H_E)}=\sup_{\alpha\in E }|x(\alpha)|\cdot \norm{Q_{\alpha}}_{op}^{\frac{1-\theta(\alpha)}{2}},\\
    &\norm{T^{-1}}_{B(H_E)}=\norm{(UT)^{-1}}_{B(H_E)}=\sup_{\alpha\in E }|x(\alpha)^{-1}|\cdot \norm{Q_{\alpha}^{-1}}_{op}^{\frac{1-\theta(\alpha)}{2}}.
\end{align}
Thus, we can conclude that
\begin{align}
    \sup_{\alpha\in E}\left ( \norm{Q_{\alpha}}_{op}\norm{Q_{\alpha}^{-1}}_{op} \right )^{\frac{1-\theta(\alpha)}{2}}&\leq \left ( \sup_{\alpha\in E}|x(\alpha)|\cdot \norm{Q_{\alpha}}_{op}^{\frac{1-\theta(\alpha)}{2}} \right ) \cdot \left ( \sup_{\alpha\in E}|x(\alpha)^{-1}|\cdot \norm{Q_{\alpha}^{-1}}_{op}^{\frac{1-\theta(\alpha)}{2}} \right )\\
    &\leq \norm{T}_{B(H_E)}\norm{T^{-1}}_{B(H_E)}.
\end{align}
\end{proof}
The above Theorem \ref{thm2} immediately yields the following criterion.
\begin{corollary}\label{cor30}
Let $\g$ be a non-Kac compact quantum group and suppose that there exists a subset $E\subseteq\operatorname{Irr}(\g)$ such that
\begin{equation}\label{eq301}
    \sup_{\alpha\in E}\frac{\log\left ( \norm{Q_{\alpha}}_{op}\norm{Q_{\alpha}^{-1}}_{op} \right )}{\log(1+n_{\alpha})}=\infty.
\end{equation}
In this case, for any $C\in (1,2)$, let us define $\theta:E\rightarrow [0,1]$ by $\displaystyle \theta(\alpha)=1-\frac{\log(C)}{\log(1+n_{\alpha})}$. Then the non-degenerate representation $v_{E,\theta}$ in \eqref{eq02} satisfies $\norm{v_{E,\theta}}_{L^{\infty}(\g)\overline{\otimes} B(H_E)}\leq C$ and is not unitarizable.
\end{corollary}

\begin{proof}
First, by Proposition \ref{prop30}, the choice of $\theta$ shows that
\begin{align}
   \norm{v_{E,\theta}}_{L^{\infty}(\g)\overline{\otimes}B(H_E)}\leq \sup_{\alpha \in E}n_{\alpha}^{1-\theta(\alpha)}= \sup_{\alpha \in E} n_{\alpha}^{\frac{\log(C)}{\log(1+n_{\alpha})}}\leq C
\end{align}
In addition, we also have
    \begin{align}
    \log \left ( \left [ \norm{Q_{\alpha}}_{op}\norm{Q_{\alpha}^{-1}}_{op}\right ]^{1-\theta(\alpha)} \right )&=\left (1-\theta(\alpha)\right )\cdot \log \left ( \norm{Q_{\alpha}}_{op}\norm{Q_{\alpha}^{-1}}_{op} \right )\\
    &=\log(C)\cdot \frac{\log \left ( \norm{Q_{\alpha}}_{op}\norm{Q_{\alpha}^{-1}}_{op} \right )}{\log(1+n_{\alpha})}.
    \end{align}
Then the given assumption \eqref{eq301} implies
\begin{align}
\sup_{\alpha\in E}\left ( \norm{Q_{\alpha}}_{op}\norm{Q_{\alpha}^{-1}}_{op}\right )^{1-\theta(\alpha)}=\infty,
\end{align}
so we can conclude that $v_{E,\theta}$ is not unitarizable by Theorem \ref{thm2}.
\end{proof}

\begin{example}[Compact quantum groups with sub-exponential dual growth]
Suppose that $\g$ is a non-Kac compact quantum group and $\widehat{\g}$ has sub-exponential growth. Choose $\beta\in\operatorname{Irr}(\g)$ such that $\norm{Q_{\beta}}_{op}>1$. For any $k\geq 1$, the $k$-th power tensor representation has an irreducible decomposition
\begin{equation}
    \left(u^{\beta}\right)^{\otimes k}\cong\bigoplus_{\alpha\in\operatorname{Irr}(\g)}m_{\alpha,k}u^{\alpha}.
\end{equation}
In addition, by Proposition 6.1 and the discussion following Proposition 6.2 of \cite{KrSo18}, there exists an irreducible component $\alpha_k$ of $\left(u^{\beta}\right)^{\otimes k}$ such that $\norm{Q_{\alpha_k}}_{op}=\norm{Q_{\beta}}_{op}^{k}$.

Since the sub-exponential growth assumption implies $\displaystyle \lim_{k\rightarrow \infty}\frac{\log(1+n_{\alpha_k})}{k}=0$, we obtain
\begin{align}
    &\lim_{k\rightarrow \infty}\frac{\log\left(\norm{Q_{\alpha_k}}_{op}\norm{Q_{\alpha_k}^{-1}}_{op}\right)}{ \log(1+n_{\alpha_k})}\geq \lim_{k\rightarrow \infty}\frac{k\log\left(\norm{Q_{\beta}}_{op}\right)}{\log(1+n_{\alpha_k})} = \infty.
\end{align}
This implies that the subset $E=\left\{\alpha_k:k\geq 1\right\}$ satisfies \eqref{eq301}. Thus, by Corollary \ref{cor30}, we obtain an explicit version of the non-unitarizability conclusion of \cite[Theorem 3.1]{BrYo19} with non-unitarizable non-degenerate representations with norms arbitrarily close to $1$.    
\end{example}

\begin{example}[The free unitary quantum groups $U_F^+$ with $F\in \operatorname{GL}_2(\mathbb{C})$]
Let $\g=U_F^+$ be a non-Kac free unitary quantum group with $F\in\operatorname{GL}_2(\mathbb{C})$. The irreducible unitary
representations of $U_F^+$ are indexed by the free monoid $\n_0*\n_0$ generated by two letters $a$ and $b$ satisfying $\overline{a}=b$ and $\overline{b}=a$. Let $w_k$ be the alternating word of length $k$ starting with $a$:
\begin{equation}
    w_1=a,    \qquad w_2=ab, \qquad w_3=aba, \qquad \cdots.
\end{equation}
Then $n_{w_k}=k+1$ and $\norm{Q_{w_k}}_{op}=\norm{Q_{w_k}^{-1}}_{op}=\norm{Q_{w_1}}^k$, and hence
\begin{equation}
    \lim_{k\rightarrow \infty}\frac{ \log\left(\norm{Q_{w_k}}_{op}\norm{Q_{w_k}^{-1}}_{op}\right)}{\log(1+n_{w_k})}=\lim_{k\rightarrow \infty}\frac{2k\log(\norm{Q_{w_1}})}{\log(k+2)}=\infty.
\end{equation}
Thus, the proper subset $E=\left\{w_k:k\geq 1\right\} \subseteq\operatorname{Irr}(U_F^+)$ satisfies \eqref{eq301}. Note that this example is not covered by the sub-exponential growth assumption.    
\end{example}

\subsection{Non-unitarizable non-degenerate representations of discrete quantum groups}

We begin by recalling a lifting principle from the classical setting. Let $G$ and $H$ be locally compact groups and suppose that $H$ is a quotient of $G$, i.e. there is a continuous surjective homomorphism $q:G\longrightarrow H$. If $\pi:H\longrightarrow B(K)$ is a uniformly bounded representation, then
\begin{equation}
\widetilde{\pi}=\pi\circ q:G\longrightarrow B(K)    
\end{equation}
is a uniformly bounded representation satisfying $\|\widetilde{\pi}\|=\|\pi\|$. Moreover, $\widetilde{\pi}$ is
unitarizable if and only if $\pi$ is unitarizable. Thus, non-unitarizable representations of $H$ can be lifted along quotient group homomorphisms without changing their norms.

Within the framework of locally compact quantum groups, $H$ is a quotient of $G$ if and only if $\widehat{H}$ is a closed quantum subgroup of $\widehat{G}$ \cite{Vae05,DaKaSkSo12}. For general locally compact quantum groups, we say $\widehat{\mathbb H}$ is a closed quantum subgroup of $\widehat{\mathbb G}$ in the sense of Vaes,  if there exists an injective normal unital $*$-homomorphism $\gamma:L^\infty(\mathbb H)\longrightarrow L^\infty(\mathbb G)$ satisfying $\Delta_{\mathbb G}\circ\gamma=(\gamma\otimes\gamma)\circ\Delta_{\mathbb H}$ by \cite[Theorem 3.3]{DaKaSkSo12}. The preadjoint of $\gamma$ is the surjective map $\gamma_*:L^1(\mathbb G)\longrightarrow L^1(\mathbb H)$ given by $\gamma_*(\omega)=\omega\circ\gamma$ for all $\omega\in L^1(\mathbb G)$. The following proposition is therefore a quantum analogue of the lifting principle above.

\begin{proposition}\label{prop430}
Let $\widehat{\mathbb H}$ be a closed quantum subgroup of a locally compact quantum group $\widehat{\mathbb G}$ in the sense of Vaes, and let $\gamma: L^\infty(\mathbb H) \longrightarrow L^\infty(\mathbb G)$ be the associated injective normal unital $*$-homomorphism. If $V\in L^\infty(\mathbb H)\overline{\otimes}B(K)$ is a non-degenerate representation of $\mathbb H$, then $\widetilde V =(\gamma\otimes\id)(V)\in L^\infty(\mathbb G)\overline{\otimes}B(K)$ is a non-degenerate representation of $\mathbb G$ satisfying $\norm{\widetilde V}=\norm{V}$. Moreover, $\widetilde V$ is unitarizable if and only if $V$ is unitarizable.
\end{proposition}

\begin{proof}
Since $\gamma$ intertwines the comultiplications,
\begin{align*}
    (\Delta_{\mathbb G}\otimes\id)(\widetilde V)=
    (\gamma\otimes\gamma\otimes\id)\bigl((\Delta_{\mathbb H}\otimes\id)(V)\bigr) =
    (\gamma\otimes\gamma\otimes\id)(V_{13}V_{23}) =\widetilde V_{13}\widetilde V_{23}.
\end{align*}
Thus, $\widetilde V$ is a representation of $\mathbb G$. Since $\gamma\otimes\id$ is an injective $*$-homomorphism, it is
isometric, and hence $\|\widetilde V\|=\|V\|$. Moreover, since $(\omega\otimes\id)(\widetilde V)=(\gamma_*(\omega)\otimes\id)(V)$ for all $\omega\in L^1(\mathbb G)$ and $\gamma_*:L^1(\g)\rightarrow L^1(\mathbb{H})$ is surjective, the non-degeneracy of $V$ implies that $\widetilde V$ is also non-degenerate.

If $V$ is unitarizable, then there exists an invertible operator $T\in B(K)$ such that $V_T=(1\otimes T)V(1\otimes T^{-1})$ is unitary. This implies that $(\gamma\otimes\id)(V_T)=(1\otimes T)\widetilde V(1\otimes T^{-1})$ is also unitary since  $\gamma\otimes\id$ is a unital $*$-homomorphism. Thus, $\widetilde{V}$ is unitarizable. 

Conversely, suppose that there exists an invertible operator $T\in B(K)$ such that $\widetilde V_T=(1\otimes T)\widetilde V(1\otimes T^{-1})$ is unitary. Let $V_T=(1\otimes T)V(1\otimes T^{-1})$. Then, since
\begin{align}
(\gamma\otimes \operatorname{id})(V_T^*V_T)=(\gamma\otimes \operatorname{id})(V_T)^*(\gamma\otimes \operatorname{id})(V_T)=\widetilde{V}_T^*\widetilde{V}_T=1\otimes \operatorname{Id}_K\\
(\gamma\otimes \operatorname{id})(V_TV_T^*)=(\gamma\otimes \operatorname{id})(V_T)(\gamma\otimes \operatorname{id})(V_T)^*=\widetilde{V}_T\widetilde{V}_T^*=1\otimes \operatorname{Id}_K
\end{align}
and $\gamma\otimes\id$ is an injective unital $*$-homomorphism, it follows that $V_T$ is unitary. Hence $V$ is unitarizable.
\end{proof}

We recall that there is another notion of a closed quantum subgroup, due to Woronowicz. For compact or discrete quantum groups, the notions of closed quantum subgroup in the senses of Vaes and Woronowicz coincide \cite{DaKaSkSo12}. Therefore, throughout the compact and discrete settings considered below, we simply use the term closed quantum subgroup, without specifying the notion.

\begin{example}[Direct products]
Let $\mathbb G$ and $\mathbb H$ be locally compact quantum groups. Then $\widehat{\mathbb G}$ and $\widehat{\mathbb H}$ are closed quantum subgroups of $\widehat{\mathbb G}\times \widehat{\mathbb H}=\widehat{\mathbb G\times\mathbb H}$ in the sense of Vaes. Consequently, if either $\mathbb G$ or $\mathbb H$
admits a non-unitarizable non-degenerate representation, then $\mathbb G\times\mathbb H$ admits a non-unitarizable non-degenerate representation with the same norm.
\end{example}

\begin{example}[Free products]\label{ex-free-product}

Let $\mathbb G_1$ and $\mathbb G_2$ be either both compact quantum groups or both discrete quantum groups. Their free product $\mathbb G_1*\mathbb G_2$ is well defined in both settings: in the compact case, we use Wang's free-product construction \cite{Wa95}, while in the discrete case the free product is defined by duality; see \cite{DaFiSkWh16}. With this convention, in either case we have
\begin{equation}\label{eq440}
    \widehat{\mathbb G_1*\mathbb G_2}=\widehat{\mathbb G_1}*\widehat{\mathbb G_2}.
\end{equation}

Each $\mathbb G_i$ is a closed quantum subgroup of $\mathbb G_1*\mathbb G_2$. Hence, by Proposition \ref{prop430}, any non-unitarizable non-degenerate representation of $\widehat{\mathbb G_i}$ lifts to a non-unitarizable non-degenerate representation of $\widehat{\mathbb G_1*\mathbb G_2}=\widehat{\mathbb G_1}*\widehat{\mathbb G_2}$  with the same norm. On the other hand, the above identity \eqref{eq440} shows that each $\widehat{\mathbb G_i}$ is a closed quantum subgroup of $\widehat{\mathbb G_1}*\widehat{\mathbb G_2}=\widehat{\mathbb G_1*\mathbb G_2}$. Thus, by Proposition \ref{prop430}, any non-unitarizable non-degenerate representation of $\mathbb G_i$ lifts to a non-unitarizable non-degenerate representation of $\mathbb G_1*\mathbb G_2$, again with the same norm.
\end{example}

To apply Proposition \ref{prop430} to $\mathbb{F}U_F$ and $\mathbb{F}O_F$, we first present two auxiliary results from classical group theory. We begin with the concept of induction from an open subgroup. Let $G$ be a locally compact group, $H$ be an open subgroup of $G$, and $\pi:H \longrightarrow B(H_\pi)$ be a strongly continuous, uniformly bounded representation. For unitary representations, the standard induction construction from an open subgroup is described, for example, in \cite[Section 2.1]{KaTa13}. 

In the discrete setting, Pisier noted that Mackey induction can lift a uniformly bounded representation of $H$ to a uniformly bounded representation of $G$, while preserving both the norm and unitarizability; see \cite[Proposition 0.5(i)]{Pis05}. In the general setting, let $G$ be a locally compact group, $H$ be an open subgroup of $G$, and let $\pi:H\longrightarrow B(H_\pi)$ be a strongly continuous uniformly bounded representation. Fix a section $s:G/H\longrightarrow G$ satisfying $s(H)=e$, and define $c(g,x)=s(gx)^{-1}gs(x)$ for all $g\in G$ and $x\in G/H$. For each $g\in G$, the map
\begin{equation}
\Pi_s(g)(\delta_x\otimes\xi)=\delta_{gx}\otimes\pi(c(g,x))\xi, \qquad x\in G/H,\quad \xi\in H_\pi,
\end{equation}
extends uniquely to a bounded operator on $\ell^2(G/H)\otimes H_\pi\cong \ell^2(G/H,H_{\pi})$. This defines a strongly continuous uniformly bounded representation $\Pi_s:G\longrightarrow B\left(\ell^2(G/H)\otimes H_\pi\right)$, which is non-unitarizable if $\pi$ is not unitarizable \cite[Proposition 6.3.1]{Ghe17}. We record some additional properties needed for our purposes in the following lemma.
\begin{lemma}\label{openinduction}
The induced representation $\Pi_s$ above satisfies $\norm{\Pi_s}=\norm{\pi}$, and $\Pi_s$ is unitarizable if and only if $\pi$ is unitarizable.
\end{lemma}

\begin{proof}

We first verify the norm equality. For $F=(F(x))_{x\in G/H}\in\ell^2(G/H,H_\pi)$, we have
\begin{equation}
(\Pi_s(g)F)(gx)=\pi(c(g,x))F(x)
\end{equation}
for all $x\in G/H$. Therefore,
\begin{align}
\norm{\Pi_s(g)F}^2=\sum_{x\in G/H}\norm{\pi(c(g,x))F(x)}^2 \leq \norm{\pi}^2\sum_{x\in G/H}\norm{F(x)}^2=\norm{\pi}^2\norm{F}^2.
\end{align}
Thus, we obtain $\norm{\Pi_s}\leq\norm{\pi}$. On the other hand, since $s(H)=e$, we have $c(h,H)=h$ for all $h\in H$, and hence
\begin{equation}\label{eq332}
\Pi_s(h)(\delta_H\otimes\xi)=\delta_H\otimes\pi(h)\xi, \qquad h\in H,\quad \xi\in H_\pi.
\end{equation}
Thus, $\delta_H\otimes H_\pi$ is invariant under $\left.\Pi_s\right|_H$, and the corresponding subrepresentation is unitarily equivalent to $\pi$. Hence, we obtain $\norm{\Pi_s}\geq\norm{\pi}$.

It remains to prove the unitarizability part. Suppose first that $\pi$ is unitarizable. Then there exists an invertible operator $T\in B(H_\pi)$ such that $\pi_T=T^{-1}\pi(\cdot )T$ is a unitary representation of $H$. For $g\in G$, $x\in G/H$, and $\xi\in H_\pi$, we have
\begin{align}
[(I\otimes T)^{-1}\Pi_s(g)(I\otimes T)](\delta_x\otimes\xi)=\delta_{gx}\otimes T^{-1}\pi(c(g,x))T\xi =\delta_{gx}\otimes\pi_T(c(g,x))\xi.
\end{align}
Since $x\mapsto gx$ is a permutation of $G/H$ and $\pi_T(c(g,x))$ is unitary for every $x\in G/H$, for any $F=(F(x))_{x\in G/H}\in\ell^2(G/H,H_\pi)$ we have
\begin{align}
\norm{(I\otimes T)^{-1}\Pi_s(g)(I\otimes T)F}^2=\sum_{x\in G/H}\norm{\pi_T(c(g,x))F(x)}^2=\sum_{x\in G/H}\norm{F(x)}^2=\norm{F}^2.
\end{align}
Thus $(I\otimes T)^{-1}\Pi_s(g)(I\otimes T)$ is an isometry. Since it is invertible, with inverse $(I\otimes T)^{-1}\Pi_s(g^{-1})(I\otimes T)$, it is unitary. Hence $\Pi_s$ is unitarizable.

Conversely, suppose that $\Pi_s$ is unitarizable. Then $\left.\Pi_s\right|_H$ is unitarizable. By \eqref{eq332}, the subspace $\delta_H\otimes H_\pi$ is invariant under $\left.\Pi_s\right|_H$, and the corresponding subrepresentation is unitarily equivalent to $\pi$. Hence $\pi$ is unitarizable.
\end{proof}

By \cite{Wys93}, non-abelian free groups $\mathbb{F}_m=\z^{*m}$, $m\geq 2$, admit non-unitarizable uniformly bounded representations with norm arbitrarily close to $1$. The following lemma shows that the same conclusion holds for $\Gamma_{m,n}=\mathbb{Z}^{*m}*\mathbb{Z}_2^{*n}$ whenever $2m+n\geq 3$.

\begin{lemma}\label{lem330}
Let $m,n\in\mathbb N_0$ such that $2m+n\geq 3$. Then $\Gamma_{m,n}=\mathbb Z^{*m}*\mathbb Z_2^{*n}$ contains a subgroup isomorphic to $\mathbb F_2$. In particular, for any $\epsilon>0$, there exists a non-unitarizable uniformly bounded representation $\pi:\Gamma_{m,n}\longrightarrow B(H_\pi)$ such that $\norm{\pi}<1+\epsilon$.
\end{lemma}

\begin{proof}
Let us show that whenever $2m+n \geq 3$, then the group 
\begin{equation}
\Gamma_{m,n}=\z^{*m}*\z_2^{*n}=\left\langle a_1,\ldots,a_m,s_1,\ldots,s_n \,\middle|\, s_1^2=\cdots=s_n^2=e\right\rangle. 
\end{equation}
contains a copy of $\mathbb F_2$.
\begin{itemize}
\item If $m\geq 2$,  then $\mathbb{F}_2\leq \mathbb{F}_m=\mathbb{Z}^{*m} \leq \Gamma_{m,n}.$
\item If $m=1$, then $n\geq 1$, and hence $\Gamma_{1,n}$ contains
\[
\mathbb{Z}*\mathbb{Z}_2
 =\langle a,s\mid s^2=e\rangle.
\]
The subgroup generated by $a$ and $sas$ is free of rank two, and therefore is isomorphic to $\mathbb{F}_2$.
\item If $m=0$, then $n\geq 3$, and $\Gamma_{0,n}$ contains
\[
\mathbb{Z}_2*\mathbb{Z}_2*\mathbb{Z}_2
 =\langle s_1,s_2,s_3\mid s_1^2=s_2^2=s_3^2=e\rangle.
\]
The subgroup generated by $s_1s_2$ and $s_1s_3$ is free of rank two.
\end{itemize}
Hence, the existence of a non-unitarizable uniformly bounded representation whose norm is less than $1+\epsilon$ follows directly from Lemma \ref{openinduction} and \cite{Wys93}.

\end{proof}

We now apply Proposition \ref{prop430}, Lemma \ref{openinduction} and Lemma \ref{lem330} to construct non-unitarizable non-degenerate representations of $\mathbb FU_F$ and $\mathbb FO_F$. This result is new even in the Kac case, namely for $\mathbb F U_N$ and $\mathbb F O_N$.

\begin{theorem}\label{thm-main-3}
Let $\epsilon>0$. The following assertions hold.
\begin{enumerate}
    \item Let $N\geq 2$ and $F\in\operatorname{GL}_N(\Comp)$. Every unitary free quantum group $\mathbb FU_F$ admits an explicit non-unitarizable non-degenerate representation of norm less than $1+\epsilon$.
    \item Let $N\geq 3$ and $F\in\operatorname{GL}_N(\Comp)$ satisfy $F\overline{F}\in\mathbb R\,\operatorname{Id}_N$. Every orthogonal free quantum group $\mathbb FO_F$ admits an explicit non-unitarizable non-degenerate representation of norm less than $1+\epsilon$.
\end{enumerate}
\end{theorem}

\begin{proof}
By Lemma \ref{lem330} and Proposition \ref{prop430}, it is enough to show that their dual compact quantum groups $U_F^+$ and $O_F^+$ contain a closed quantum subgroup of the form $\widehat{\Gamma_{m,n}}$, for some $m,n\in\mathbb N_0$ satisfying $2m+n\geq 3$. We will repeatedly use the following closed quantum subgroup inclusions: $\widehat{\mathbb F_n}\subseteq U_F^+$, $F\in\operatorname{GL}_n(\Comp)$, and $\widehat{\mathbb Z_2^{*n}}\subseteq O_n^+$ as closed quantum subgroups; see \cite{Ban97,DaKaSkSo12} and \cite{BaPa17,SpWe19} for more details. We also note that $\widehat{\mathbb F_n}\subseteq O_{J_n}^+$ where $J_n=\begin{bmatrix}
0&\operatorname{Id}_n\\
-\operatorname{Id}_n&0
\end{bmatrix}$. Indeed, if $\operatorname{diag}(u_1,\cdots,u_{n})$ denotes the fundamental unitary representation of $\widehat{\mathbb{F}_{n}}$, then 
\begin{equation}
U=\operatorname{diag}(u_1,\ldots,u_{n},u_1^*,\ldots,u_{n}^*)    
\end{equation}
is unitary and satisfies the defining relation $U=J_nU^cJ_n^{-1}$ of $O_{J_n}^+$. Thus the universal property of $O_{J_n}^+$ implies $\widehat{\mathbb F_{n}}\subseteq O_{J_n}^+$.

(1) We first consider the unitary free quantum groups. Recall that $\widehat{\mathbb{F}_N}\subseteq U_F^+$. Since $\mathbb{F}_N=\Gamma_{N,0}$ and $2m+n=2N+0=2N\geq 4$, by Lemma \ref{lem330}, for any $\epsilon>0$, there exists a non-unitarizable uniformly bounded non-degenerate representation of $\mathbb F_N$ with norm less than $1+\epsilon$. Therefore, Proposition \ref{prop430} lifts this representation to a non-unitarizable non-degenerate representation of $\mathbb FU_F$ with the same norm.

(2) We now consider the orthogonal free quantum groups. We use the description of the maximal Kac quantum subgroup of $O_F^+$ obtained in \cite[Theorems 3.3 and 3.4]{DaFrSk21}. Up to isomorphism, after the canonical normalization of $F$, there are two cases according as $F\overline F=\operatorname{Id}_N$ or $F\overline F=-\operatorname{Id}_N$.

Suppose first that $F\overline F=\operatorname{Id}_N$. If $F^*F=\operatorname{Id}_N$, then $O_F^+$ is isomorphic to $O_N^+$, which contains $\widehat{\mathbb Z_2^{*N}}=\widehat{\Gamma_{0,N}}$ as a closed quantum subgroup. Hence the desired conclusion follows from Lemma \ref{lem330} and Proposition \ref{prop430}.

Suppose now that $F^*F\neq \operatorname{Id}_N$. Let $0<q_1<\cdots<q_r<1$ be such that the eigenvalues of $F^*F$
different from $1$ are $q_i^2$ and $q_i^{-2}$ ($1\leq i\leq r)$, each occurring with multiplicity $m_i$. Then the maximal Kac quantum subgroup of $O_F^+$ is
\begin{equation}\label{eq307}
    U_{m_1}^+*\cdots*U_{m_r}^+*O_{N-2K}^+,
\end{equation}
where $K=m_1+\cdots+m_r$ and where the last factor is ignored when $N-2K=0$. By Example \ref{ex-free-product}, it follows that $\widehat{\mathbb{F}_{m_1}}*\cdots*\widehat{\mathbb{F}_{m_r}}*\widehat{\mathbb Z_2}^{*(N-2K)}=\widehat{\Gamma_{K,N-2K}}$ is a closed quantum subgroup of $O_F^+$. Since $2K+(N-2K)=N\geq 3$, the desired conclusion follows.

Suppose next that $F\overline F=-\operatorname{Id}_N$. In this case, $N$ is even and $N\geq 4$. Let
\begin{equation}
    0<q_1<\cdots<q_{r-1}<q_r=1
\end{equation}
be such that, for $1\leq i\leq r-1$, $q_i^2$ and $q_i^{-2}$ are eigenvalues of $F^*F$, each occurring with multiplicity $m_i$, while the eigenvalue $1$ occurs with multiplicity $2m_r$. Then $m_1+\cdots+m_r=\frac{N}{2}$. Here, $m_r=0$ if $1$ is not an eigenvalue of $F^*F$. 

If $m_r=0$, then the maximal Kac quantum subgroup of $O_F^+$ is $U_{m_1}^+*\cdots*U_{m_{r-1}}^+$. Since $m_1+\cdots+m_{r-1}=\frac{N}{2}$, Example \ref{ex-free-product} shows that $\widehat{\mathbb{F}_{m_1}}*\cdots *\widehat{\mathbb{F}_{m_{r-1}}}=\widehat{\mathbb F_{N/2}}=\widehat{\Gamma_{N/2,0}}$ is a closed quantum subgroup of $O_F^+$. Hence the desired conclusion follows since $\frac{N}{2}\geq 2$. 

If $m_r\geq 1$, then the maximal Kac quantum subgroup of $O_F^+$ is
\begin{equation}\label{eq308}
    U_{m_1}^+*\cdots*U_{m_{r-1}}^+*O_{J_{m_r}}^+,
\end{equation}
where $J_{m_r}=\begin{bmatrix}0 & \operatorname{Id}_{m_r}\\ -\operatorname{Id}_{m_r} & 0\end{bmatrix}$, and it follows that $\widehat{\mathbb F_{m_1}}*\cdots*\widehat{\mathbb F_{m_r}}=\widehat{\mathbb F_{N/2}}=\widehat{\Gamma_{N/2,0}}$ is a closed quantum subgroup of $O_F^+$ by Example \ref{ex-free-product}. Hence, the desired conclusion follows since $\frac{N}{2} \geq 2$.
\end{proof}

\subsection*{Acknowledgements}
F.K. was supported by the Iran National Science Foundation(INSF) under Project No. 40407519. S.-G.Y. was supported by the National Research Foundation of Korea (NRF) grant funded by the Korea government(MSIT) (No.RS-2025-00561391 and No.RS-2024-00413957). 

\bibliographystyle{myalpha}
\bibliography{Youn}

\end{document}